\documentclass[11pt]{amsart}	%This version: VB-EP March 17, 2023

\usepackage[utf8]{inputenc}
\usepackage{graphicx}
\usepackage{subcaption}
\usepackage{amsmath}
\usepackage{amssymb}
\usepackage{amsthm}
\usepackage{ifthen} 
\usepackage{color}

\usepackage{dsfont}

\def\polk#1{\setbox0=\hbox{#1}{\ooalign{\hidewidth\lower1.5ex\hbox{`}\hidewidth\crcr\unhbox0}}}

\newtheorem{Theorem}{Theorem}

\newtheorem{Lemma}{Lemma}

\newtheorem{Proposition}{Proposition}

\theoremstyle{definition}

\theoremstyle{remark}
\newtheorem{Remark}{Remark}[section]

\allowdisplaybreaks

\newcommand{\intav}[1]{\mathchoice {\mathop{\vrule width 6pt height 3 pt depth -2.5pt\kern -8pt \intop}\nolimits_{\kern -6pt#1}} {\mathop{\vrule width5pt height 3 pt depth -2.6pt \kern -6pt \intop}\nolimits_{#1}}{\mathop{\vrule width 5pt height 3 pt depth -2.6pt \kern -6pt\intop}\nolimits_{#1}} {\mathop{\vrule width 5pt height 3 pt depth-2.6pt \kern -6pt \intop}\nolimits_{#1}}}

\usepackage{setspace}
\numberwithin{equation}{section}

\title[Hessian-dependent functionals]{Improved regularity for a\\ Hessian-dependent functional}

\author[V. Bianca]{Vincenzo Bianca}
\address{University of Coimbra, CMUC, Department of Mathematics, 3001-501 Coimbra, Portugal}{}
\email{vincenzo@mat.uc.pt}

\author[E. A. Pimentel]{Edgard A. Pimentel}
\address{University of Coimbra, CMUC, Department of Mathematics, 3001-501 Coimbra, Portugal and Pontifical Catholic University of Rio de Janeiro -- PUC-Rio, 22451-900 G\'avea, Rio de Janeiro-RJ, Brazil}{}
\email{edgard.pimentel@mat.uc.pt}

\author[J.M. Urbano]{Jos\'{e} Miguel Urbano}
\address{King Abdullah University of Science and Technology (KAUST), Computer, Electrical and Mathematical Sciences and Engineering Division (CEMSE), Thuwal 23955-6900, Saudi Arabia and University of Coimbra, CMUC, Department of Mathematics, 3001-501 Coimbra, Portugal}{} 
\email{miguel.urbano@kaust.edu.sa} 

\date{\today}

\begin{document}
	
\begin{abstract}
We prove that minimizers of the $L^{d}$-norm of the Hessian in the unit ball of $\mathbb{R}^d$ are locally of class $C^{1,\alpha}$. Our findings extend previous results on Hessian-dependent functionals to the borderline case and resonate with the H\"older regularity theory available for elliptic equations in double-divergence form.
\end{abstract}

	
	
\maketitle
	
\section{Introduction}\label{sec_mollybloom}
	
We consider the Hessian-dependent functional $I:W^{2,d}(B_1)\to\mathbb{R}$ given by
\begin{equation}\label{eq_main1}
	I(u):=\int_{B_1}\left|D^2u\right|^d\,{\rm d}x,
\end{equation}
where $B_1\subset\mathbb{R}^d$ is the unit ball in the Euclidean space $\mathbb{R}^d$, and examine the regularity of minimizers for $I$ in H\"older spaces. In particular, we prove that minimizers are locally of class $C^{1,\alpha}$, with estimates.

Hessian-dependent functionals appear in various disciplines, in the realms of differential geometry, the calculus of variations, the mechanics of solids and mean-field games theory; see, for instance,  \cite{AliceChang1, AliceChang2, Wang1, Wang2, Oh93, SW03, Bhattacharya_2022, kohn2, kohn3, venkataramani, kohnicm, avilesgiga2, avilesgiga1}. A fundamental example concerns the model-problem
\[
	I_\Delta(u):=\int_{B_1}\left[{\rm Tr}(D^2u)\right]^2{\rm d}x = \int_{B_1}\left|\Delta u\right|^2{\rm d}x,
\]
whose first compactly supported variation yields the biharmonic operator and drives the so-called biharmonic maps. The functional $I_\Delta$ resonates in the analysis of conformally invariant energies since it is conformally invariant in dimension $d=4$. For an analysis of biharmonic maps targeting the $m$-dimension sphere $S^m$, we refer the reader to \cite{AliceChang1, AliceChang2}. See also \cite{Wang1, Wang2} for biharmonic maps into Riemannian manifolds. 

In the context of differential geometry, Hessian-dependent functionals arise in the study of Lagrangian surfaces minimizing the area
\begin{equation}\label{eq_slimani}
	A(u):=\int_{B_1}\sqrt{I_d+(D^2u)^2}\,{\rm d}x,
\end{equation}
where $I_d$ is the identity matrix of order $d$. The Euler-Lagrange equation associated with the functional $A$ is the double-divergence form pde
\[
	\left(\sqrt{{\rm det}(I_d+(D^2u)^2}(I_d+(D^2u)^2)_{i,j}\delta^{k,\ell}u_{i,k}\right)_{x_jx_\ell}=0\hspace{.3in}\mbox{in}\hspace{.3in}B_1,
\]
where $\delta^{k,\ell}$ is Kronecker's delta. For results on \eqref{eq_slimani}, we mention \cite{Oh93,SW03}. We also highlight the set of results put forward in \cite{Bhattacharya_2022}. In that paper, the author considers
\begin{equation}\label{eq_blackwell}
	L(u):=\int_{B_1}F(D^2u)\,{\rm d}x,
\end{equation}
where $F$ is convex and smooth, and its main contribution concerns the regularity of minimizers. Indeed, it is proven therein that if $u\in W^{2,\infty}_{\rm loc}(B_1)$ minimizes the energy in \eqref{eq_blackwell} and its Hessian satisfies a small-oscillation condition, then $u$ is of class $C^{2,\alpha}$. The argument here resembles the proof of the Evans-Krylov theory put forward in \cite[Chapter 6]{ccbook}. Once $C^{2,\alpha}$-regularity is available, the author proceeds by proving that solutions are indeed in $C^{\infty}_{\rm loc}(B_1)$.

One also finds applications of Hessian-dependent functionals in the context of the mechanics of solids \cite{kohnicm}. The usual examples concern energy-driven pattern formation and nonlinear elasticity. Typically, these models examine wrinkles appearing in twisted ribbon or blister patterns in thin films on compliant substrates \cite{kohn2, kohn3,venkataramani}. 

In the context of the calculus of variations, Hessian-dependent functionals also play a role. The work-horse of the theory is the so-called Aviles-Giga functional \cite{avilesgiga2,avilesgiga1}, given by
\begin{equation}\label{eq_bodleian}
	G_\varepsilon(u):=\int_{B_1}\varepsilon^{-1}\left(1-\left|Du\right|^2\right)^2+\varepsilon|D^2u|^2\,{\rm d}x,
\end{equation}
for $\varepsilon>0$. This functional can be regarded as a natural generalization of the Modica-Mortola functional to the context of higher-order terms (see \cite{Modica-Mortola}). In addition, it relates to the distance function (to the boundary of a domain) and, naturally, with the solutions of the eikonal equation. Indeed, in \cite{JaOtPe} the authors replace the unit ball with a general, bounded domain $\Omega\subset \mathbb{R}^d$ and show that if there exist sequences $(u_n)_{n\in\mathbb{N}}$ and $(\varepsilon_n)_{n\in\mathbb{N}}$ satisfying
\[
	\int_{\Omega}\varepsilon_n^{-1}\left(1-\left|Du_n\right|^2\right)^2+\varepsilon_n|D^2u_n|^2\,{\rm d}x\longrightarrow 0
\]
as $\varepsilon\to 0$, then $\Omega$ has to be a ball. Moreover, they conclude that 
\[
	\lim_{n\to\infty} u_n(x)={\rm dist}(x,\partial \Omega).
\]
The functional $G_\varepsilon$ also appears in connection with problems in thin film blisters \cite{GiOr97} and liquid crystals \cite{avilesgiga2}.

The regularity of the minimizers for a functional similar to \eqref{eq_main1} was studied in \cite{Andrade_Pimentel2021}. In that article, the authors examine functionals of the form
\[
	J(u):=\int_{B_1}\left|F(D^2u)\right|^p\,{\rm d}x,
\] 
where $F:\mathcal{S}(d)\sim\mathbb{R}^\frac{d(d+1)}{2}\to\mathbb{R}$ satisfies the condition 
\[
	\frac{1}{\lambda}\left|M\right|\leq F(M)\leq \lambda\left|M\right|,
\]
for every symmetric matrix $M\in S(d)$, and some constant $\lambda>1$. Under the assumption $p>d$, they prove the gradient of minimizers is $C^{0,(p-d)/(p-1)}$-regular, with estimates. However, the case $p=d$ falls off the scope of the results in \cite{Andrade_Pimentel2021}, and we treat it here, for simplicity of exposition, for the model case $F=I_d$ (see Remark \ref{tfa+yy}).

We examine the functional in \eqref{eq_main1} and establish a regularity result for minimizers $u\in W^{2,d}(B_1)\cap W^{1,d}_{g}(B_1)$, where $g\in W^{2,d}(B_1)$ is a boundary condition attained in the Sobolev sense. Our main result reads as follows.

\begin{Theorem}[$C^{1,\alpha}$-regularity estimates]\label{thm_ll}
Let $u\in W^{2,d}(B_1)\cap W^{1,d}_g(B_1)$ be a minimizer for \eqref{eq_main1}, where $g\in W^{2,d}(B_1)$ is given. Then there exists $\alpha\in(0,1)$ such that $u\in C_{{\rm loc}}^{1,\alpha}(B_1)$. In addition, there exists a constant $C>0$, depending only on the dimension $d$, such that 
\[
	[Du]_{C^{0,\alpha}(B_{1/2})}\leq C.
\]
\end{Theorem}

The proof is based on testing the Euler-Lagrange equation associated with the functional \eqref{eq_main1} against a suitable test function built upon a smooth cut-off satisfying uniform bounds up to its second derivatives. This allows us to establish a uniform decay rate for the $L^d$-norm of the Hessian of minimizers in balls of comparable radii, by extending Widman's \emph{hole-filling technique} (see \cite{Widman}) to deal with the difficulties posed by the presence of second-order derivatives.  Once the information on the decay is available, an application of Morrey's characterization of H\"older continuity completes the proof.

The remainder of this article is organized as follows. In Section \ref{sec_blenheim} we gather preliminary material used in the paper, including a discussion on the existence and uniqueness of minimizers for \eqref{eq_main1}. The proof of Theorem \ref{thm_ll} is the subject of Section \ref{sec_abel}.

\section{Preliminaries}\label{sec_blenheim}

In the sequel, we state our problem rigorously, recall preliminary ingredients and comment on the existence and uniqueness of minimizers.
	
Let $B_1\subseteq\mathbb{R}^d$ denote the unit ball in $\mathbb{R}^d$, and fix $g\in W^{2,d}(B_1)$. Set $\mathcal{A}=W^{2,d}(\Omega)\cap W_g^{1,d}(\Omega)$, where
	\[
		W^{1,d}_g(B_1):=\left\lbrace u\in W^{1,d}(B_1)\,|\,u-g\in W^{1,d}_0(B_1)\right\rbrace.
	\]
	Let $I:\mathcal{A}\to\mathbb{R}$ be defined as
	\begin{equation*}
		I(w)=\int_{B_1}|D^2w|^d\,{\rm d}x.
	\end{equation*}
	We consider the problem of finding $u\in\mathcal{A}$ such that
	\begin{equation}\label{1}
		I(u)=\min_{w\in\mathcal{A}}I(w).
	\end{equation}

We notice the first compactly supported variation of the functional $I(w)$ yields the fourth-order Euler-Lagrange equation 
\begin{equation}\label{eq_eulerlagrange}
	\left(|D^2u|^{d-2}\frac{\partial^2 u}{\partial_{x_i}\partial_{x_j}}\right)_{x_ix_j}=0\hspace{.3in}\mbox{in}\hspace{.3in}B_1.
\end{equation}
The weak form of \eqref{eq_eulerlagrange} is given by
\begin{equation}\label{eq_weakform}
	\int_{B_1}|D^2u|^{d-2}D^2u:D^2\varphi\,{\rm d}x=0\hspace{.3in}\forall\varphi\in C^\infty_c(B_1),
\end{equation}
where, for matrices $A:=\left(a_{i,j}\right)_{i,j=1}^d$ and $B:=\left(b_{i,j}\right)_{i,j=1}^d$, the operation $A:B$ stands for
\[
	A:B:=\sum_{i,j=1}^da_{i,j}b_{i,j}.
\]
Compare \eqref{eq_eulerlagrange} and \eqref{eq_weakform} with the fourth-order model studied in \cite{Bhattacharya_2022}. 

We proceed by recalling a preliminary result used further in the paper.	

 \begin{Lemma}\label{lem_stima119}
		Fix $R_0>0$ and let $\phi:[0,R_0]\to[0,\infty)$ be a non-decreasing function. Suppose there exist constants $C_1,\alpha,\beta>0$, and $C_2,\mu\ge0$, with $\beta<\alpha$, satisfying
		\begin{equation*}
			\phi(r)\le C_1\Big[\Big(\frac{r}{R}\Big)^{\alpha}+\mu\Big]\phi(R)+C_2R^{\beta},
		\end{equation*}
		for every $0<r\le R\le R_0$.
		Then, for every $\sigma\le\beta$, there exists $\mu_0=\mu_0(C_1,\alpha,\beta,\sigma)$ such that, if $\mu<\mu_0$, for every $0<r\le R\le R_0$, we have
		\begin{equation*}
			\phi(r)\le C_3\Big(\frac{r}{R}\Big)^{\sigma}\big(\phi(R)+C_2R^{\sigma}\big),
		\end{equation*}
		where $C_3=C_3(C_1,\alpha,\beta,\sigma)>0$. Moreover,
		\begin{equation*}
			\phi(r)\le C_4r^{\sigma},
		\end{equation*}
		where $C_4=C_4(C_2,C_3,R_0,\phi(R_0),\sigma)$.
	\end{Lemma}
	For the proof of Lemma \ref{lem_stima119}, we refer the reader to \cite[Lemma 2]{BiancaPimentelUrbano_2023}. The next lemma is instrumental in studying H\"older regularity.
	
	\begin{Lemma}\label{lem_vinogradoff}
		Fix $R_0>0$ and let $\phi:(0,R_0]\to[0,\infty)$ be a non-decreasing function such that for every $R\in[0,R_0]$ we have
		\begin{equation*}
			\phi(\tau R)\le\gamma\phi(R)+\sigma(R)
		\end{equation*}
	where $\sigma:(0,R_0]\to[0,\infty)$ is also non-decreasing,  $\gamma>0$ and $\tau\in(0,1)$. Then for every $\mu\in(0,1)$ and every $R\le R_0$ we have
	\begin{equation*}
		\phi(R)\le C\left[\left(\frac{R}{R_0}\right)^\alpha\phi(R_0)+\sigma\left(R^\mu R_0^{1-\mu}\right)\right]
	\end{equation*}
	where $C=C(\gamma,\tau)$ and $\alpha=\alpha(\gamma,\tau,\mu)$ are positive constants.
	\end{Lemma}
	
For a proof of Lemma \ref{lem_vinogradoff}, we refer the reader to \cite[Theorem 8.23]{Gilbarg-Trudinger}. We also recall the following characterization of H\"older continuity.

  \begin{Proposition}[Morrey's characterization of H\"older continuity] \label{morrey}
     Let $w\in W^{1,d}(B_1)$. Suppose that there exist constants $C,\beta>0$ such that
     \begin{equation*}
         \int_{B_r(x_0)}|Dw|^d\,{\rm d}x\le Cr^\beta
     \end{equation*}
     for every $B_{2r}(x_0)\subseteq B_1$. Then $w\in C_{\rm loc}^{0,\frac{\beta}{d}}(B_1)$.
 \end{Proposition}
For a proof of the previous proposition, we refer the reader to \cite[Theorem 7.19]{Gilbarg-Trudinger}. The next proposition is a key tool for estimating lower-order derivatives \cite[Theorem 5.2]{Adams_Fournier}.
	
	\begin{Proposition}\label{prop_Adams_Fournier}
		Let $\Omega\subseteq\mathbb{R}^d$ be an open set satisfying the cone condition, and $u\in W^{k,p}(\Omega)$, $k\in\mathbb{N}$. Then, for every $\varepsilon_0>0$, there exists $C>0$, depending only on $d,k,p,\varepsilon_0$ and the dimensions of the cone, such that if $\varepsilon\in(0,\varepsilon_0)$ and $j\in\{0,\dots,k\}$, then
		\begin{equation*}
			\left\|D^ju\right\|_{L^p(\Omega)}\le C\left(\varepsilon\left\|D^ku\right\|_{L^p(\Omega)}+\varepsilon^{-j/(k-j)}\|u\|_{L^p(\Omega)}\right).
		\end{equation*}
	\end{Proposition}

We close this section with a discussion on the existence and uniqueness of minimizers for the functional in \eqref{1}.

\begin{Proposition}
		There exists a unique minimizer for problem \eqref{1}.
\end{Proposition}

\begin{proof}
		Let $\left(u_m\right)_{m\in\mathbb{N}}\subseteq\mathcal{A}$ be a minimizing sequence for $I$, i.e.,
\begin{equation*}
			\lim_{m\to\infty}I(u_m)=\min_{w\in\mathcal{A}}I(w).
\end{equation*}
	Because $\left(I(u_m)\right)_{m\in\mathbb{N}}$ converges, it follows that 
	\[
		\|D^2u_m\|_{L^d(B_1)} \leq C, \quad \forall m\in\mathbb{N},
	\]
	for some $C>0$. It follows from standard Sobolev embedding results and Proposition \ref{prop_Adams_Fournier} that
	\begin{equation*}
		\|u_m\|_{L^d(B_1)}\le C\|D^2u_m\|_{L^d(B_1)}\le C^\prime
	\end{equation*}
	and
	\begin{equation*}
		\|Du_m\|_{L^d(B_1)}\le C\left(\varepsilon\|D^2u_m\|_{L^d(B_1)}+\frac{1}{\varepsilon}\|u_m\|_{L^d(B_1)}\right)\le C^\prime.
	\end{equation*}
	Hence, $\left(u_m\right)_{m\in\mathbb{N}}$ is bounded in $W^{2,d}(B_1)$. As a consequence, there exists $u_\infty\in W^{2,d}(B_1)$ such that
	\begin{equation}\label{6}
		u_m\rightharpoonup u_\infty\hspace{.3in}\mbox{in}\hspace{.3in}W^{2,d}(B_1).
	\end{equation}
	Since $\mathcal{A}$ is convex and closed, from Mazur's Theorem, it follows that $\mathcal{A}$ is weakly closed, and hence from \eqref{6} it follows that $u_\infty\in\mathcal{A}$.
	
	Since $I$ is convex and continuous in $W^{2,d}(B_1)$, it is weakly sequentially lower semi-continuous. Hence
	\begin{equation*}
		I(u_\infty)\le\lim_{m\to\infty}I(u_m)=\min_{w\in\mathcal{A}}I(w).
	\end{equation*}
	The uniqueness follows from the strict convexity of $I$.
 
	\end{proof}

\section{H\"older continuity of the gradient}\label{sec_abel}

This section details the proof of Theorem \ref{thm_ll}. We start with a technicality playing an essential role in the sequel. Let $x_0\in B_1$ and $R>0$ be such that $B_{2R}:=B(x_0,2R)\Subset B_1$. Define $\eta:B_1\to\mathbb{R}$ as
		\begin{equation*}
			\eta(x)=
			\begin{cases}
				C_\eta\exp\left(\frac{1}{|x-x_0|^2-4R^2}\right)&\hspace{.3in}\mbox{in}\hspace{.3in}B_{2R}\\
				0&\hspace{.3in}\mbox{in}\hspace{.3in}B_1\setminus B_{2R},
			\end{cases}
		\end{equation*}
  where the constant $C_\eta>0$ is chosen to ensure that $\eta$ has unit mass.
	Let us show that there exists $C>0$ such that
	\begin{equation*}
		\frac{|D\eta|^2}{|\eta|}\le C\quad\textnormal{in }B_{2R}.
	\end{equation*}

	For $i=1,\dots,d$, and for $x\in B_{2R}$ we have
	\begin{equation*}
		\frac{\partial \eta}{\partial x_i}(x)=-\frac{C_\eta2(x_i-x_{0i})}{\left(|x-x_0|^2-4R^2\right)^2}\exp\left(\frac{1}{|x-x_0|^2-4R^2}\right).
	\end{equation*}
	Hence
	\begin{align*}
		\frac{|D\eta|^2}{|\eta|}=&\frac{C_\eta4|x-x_0|^2}{\left(|x-x_0|^2-4R^2\right)^4}\exp\left(\frac{1}{|x-x_0|^2-4R^2}\right)\leq C,
	\end{align*}
	where the last inequality holds since the exponential decreases faster than the polynomial.

	\begin{proof}[Proof of Theorem \ref{thm_ll}]
		For ease of clarity, we split the proof into four steps.
		
		\medskip
		
		\noindent{\bf Step 1 - }Let $x_0\in B_1$ and $R>0$ be such that $B_{2R}:=B(x_0,2R)\Subset B_1$. Define $\zeta\in C_c^{\infty}(B_{2R})$ by
  \[
    \zeta(x):=
        \begin{cases}
            1&\hspace{.3in}\mbox{in}\hspace{.3in}B_R\\
            \overline{\zeta}(x)&\hspace{.3in}\mbox{in}\hspace{.3in}B_{3R/2}\setminus B_R\\
            \eta(x)&\hspace{.3in}\mbox{in}\hspace{.3in}B_{2R}\setminus B_{3R/2},
        \end{cases}
  \]
where $\overline\zeta$ is a smooth gluing connecting the functions in $B_R$ and $B_{2R}\setminus B_{3R/2}$. Notice that $\zeta$ is such that
		\begin{equation*}
			0\le\zeta\le1,\quad\zeta=1\quad\textnormal{in }B_R,\quad\frac{|D\zeta|^2}{|\zeta|}\le C\quad\textnormal{in }B_{2R},
		\end{equation*}
	for some constant $C>0$. Set 
 $$M=\max\left\{\|D\zeta\|_{L^{\infty}(B_{2R})},\|D^2\zeta\|_{L^{\infty}(B_{2R})}\right\}.$$
 Define $v:=\zeta^d\left(u-(u)_{B_{2R}\setminus B_R}\right)$, where
    \begin{equation*}
        (u)_{B_{2R}\setminus B_R}:=\frac{1}{|B_{2R}\setminus B_R|}\int_{B_{2R}\setminus B_R}u\,{\rm d}x.
    \end{equation*}
    We have
	\begin{equation*}
		v_{x_i}=d\zeta^{d-1}\left(u-(u)_{B_{2R}\setminus B_R}\right)\zeta_{x_i}+\zeta^du_{x_i},
	\end{equation*}
	and
	\begin{align*}
		v_{x_ix_j}
			=&d(d-1)\zeta^{d-2}\left(u-(u)_{B_{2R}\setminus B_R}\right)\zeta_{x_j}\zeta_{x_i}+d\zeta^{d-1}u_{x_j}\zeta_{x_i}\\
		&+d\zeta^{d-1}\left(u-(u)_{B_{2R}\setminus B_R}\right)\zeta_{x_ix_j}\\
		&+d\zeta^{d-1}u_{x_i}\zeta_{x_j}+\zeta^du_{x_ix_j}.
	\end{align*}
	
	\medskip
	
	\noindent{\bf Step 2 - }Testing the weak form of the Euler-Lagrange equation \eqref{eq_eulerlagrange} against the function $v$, we get
	\begin{align*}
		\int_{B_1}\zeta^d&|D^2u|^{d-2}D^2u:D^2u\,{\rm d}x\\
		=&-d(d-1)\int_{B_1}\zeta^{d-2}\left(u-(u)_{B_{2R}\setminus B_R}\right)|D^2u|^{d-2}D^2u:(D\zeta\otimes D\zeta)\,{\rm d}x\\
		&-2d\int_{B_1}\zeta^{d-1}|D^2u|^{d-2}D^2u:(Du\otimes D\zeta)\,{\rm d}x\\
		&-d\int_{B_1}\zeta^{d-1}\left(u-(u)_{B_{2R}\setminus B_R}\right)|D^2u|^{d-2}D^2u:D^2\zeta\,{\rm d}x.
	\end{align*}
	Hence
	\begin{align*}
		\int_{B_{2R}}\zeta^d|D^2u|^d\,{\rm d}x\le&C\int_{B_{2R}}\zeta^{d-2}\left|u-(u)_{B_{2R}\setminus B_R}\right||D^2u|^{d-1}|D\zeta|^2\,{\rm d}x\\
		&+C\int_{B_{2R}}\zeta^{d-1}|D^2u|^{d-1}|Du||D\zeta|\,{\rm d}x\\
		&+C\int_{B_{2R}}\zeta^{d-1}\left|u-(u)_{B_{2R}\setminus B_R}\right||D^2u|^{d-1}|D^2\zeta|\,{\rm d}x\\
		=&:I_1+I_2+I_3.
	\end{align*}
	In what follows, we estimate each of the summands $I_1$, $I_2$ and $I_3$. To estimate $I_1$, we resort to the H\"older and Poincar\'e-Wirtinger inequalities to obtain
	\begin{align*}
		I_1=&C\int_{B_{2R}}\zeta^{d-2}\left(u-(u)_{B_{2R}\setminus B_R}\right)|D^2u|^{d-1}|D\zeta|^2\,{\rm d}x\\
		=&C\int_{B_{2R}}\zeta^{d-1}\left(u-(u)_{B_{2R}\setminus B_R}\right)|D^2u|^{d-1}\zeta^{-1}|D\zeta|^2\,{\rm d}x\\
		\le&C\bigg(\int_{B_{2R}}\zeta^d|D^2u|^d\,{\rm d}x\bigg)^{1-\frac{1}{d}}\bigg(\int_{B_{2R}}\left|u-(u)_{B_{2R}\setminus B_R}\right|^d\zeta^{-d}|D\zeta|^{2d}\,{\rm d}x\bigg)^{\frac{1}{d}}\\
		\le&C\bigg(\int_{B_{2R}}\zeta^d|D^2u|^d\,{\rm d}x\bigg)^{1-\frac{1}{d}}\bigg(\int_{B_{2R}\setminus B_R}\left|u-(u)_{B_{2R}\setminus B_R}\right|^d\,{\rm d}x\bigg)^{\frac{1}{d}}\\
		\le&C\bigg(\int_{B_{2R}}\zeta^d|D^2u|^d\,{\rm d}x\bigg)^{1-\frac{1}{d}}\bigg(\int_{B_{2R}\setminus B_R}|Du|^d\,{\rm d}x\bigg)^{\frac{1}{d}}.
	\end{align*}
	To examine $I_2$, we apply H\"older's inequality to get
	\begin{align*}
		I_2=&C\int_{B_{2R}}\zeta^{d-1}|D^2u|^{d-1}|Du||D\zeta|\,{\rm d}x\\
		\le&C\bigg(\int_{B_{2R}}\zeta^d|D^2u|^d\,{\rm d}x\bigg)^{1-\frac{1}{d}}\bigg(\int_{B_{2R}}|Du|^d|D\zeta|^d\,{\rm d}x\bigg)^{\frac{1}{d}}\\
		\le&CM\bigg(\int_{B_{2R}}\zeta^d|D^2u|^d\,{\rm d}x\bigg)^{1-\frac{1}{d}}\bigg(\int_{B_{2R}\setminus B_R}|Du|^d\,{\rm d}x\bigg)^{\frac{1}{d}}.
	\end{align*}
	Finally, to estimate $I_3$, we apply once again H\"older and Poincar\'e-Wirtin\-ger inequalities to conclude that
	\begin{align*}
		I_3=&C\int_{B_{2R}}\zeta^{d-1}(u-(u)_{B_{2R}\setminus B_R})|D^2u|^{d-1}|D^2\zeta|\,{\rm d}x\\
		\le&C\bigg(\int_{B_{2R}}\zeta^d|D^2u|^d\,{\rm d}x\bigg)^{1-\frac{1}{d}}\bigg(\int_{B_{2R}}\left|u-(u)_{B_{2R}\setminus B_R}\right|^d|D^2\zeta|^d\,{\rm d}x\bigg)^{\frac{1}{d}}\\
		\le&CM\bigg(\int_{B_{2R}}\zeta^d|D^2u|^d\,{\rm d}x\bigg)^{1-\frac{1}{d}}\bigg(\int_{B_{2R}\setminus B_R}\left|u-(u)_{B_{2R}\setminus B_R}\right|^d\,{\rm d}x\bigg)^{\frac{1}{d}}\\
		\le&CM\bigg(\int_{B_{2R}}\zeta^d|D^2u|^d\,{\rm d}x\bigg)^{1-\frac{1}{d}}\bigg(\int_{B_{2R}\setminus B_R}|Du|^d\,{\rm d}x\bigg)^{\frac{1}{d}}.
	\end{align*}
	Combining the above estimates and recalling that $\zeta=1$ in $B_R$, we get
	\begin{equation}\label{eq_makeitbreakit}
		\int_{B_R}|D^2u|^d\,{\rm d}x\le C\int_{B_{2R}\setminus B_R}|Du|^d\,{\rm d}x.
	\end{equation}
	Add the quantity
	\[
		C\int_{B_R}|D^2u|^d\,{\rm d}x
	\]
	to both sides of \eqref{eq_makeitbreakit} to obtain
	\begin{equation}\label{3}
		\int_{B_R}|D^2u|^d\,{\rm d}x\le\gamma\bigg(\int_{B_{2R}}|D^2u|^d\,{\rm d}x+\int_{B_{2R}}|Du|^d\,{\rm d}x\bigg),
	\end{equation}
	with
	\begin{equation*}
		\gamma=\frac{C}{1+C}\in(0,1).
	\end{equation*}
	
	\medskip
	
	\noindent{\bf Step 3 - }Define
	\begin{equation*}
		\phi(R)=\int_{B_R}|D^2u|^d\,{\rm d}x\hspace{.3in}\mbox{and}\hspace{.3in}\sigma(R)=\int_{B_{R}}|Du|^d\,{\rm d}x,
	\end{equation*}
	and notice that both $\phi$ and $\sigma$ are non-decreasing functions. We re-write \eqref{3} as
	\begin{equation*}
		\phi(R)\le\gamma\big(\phi(2R)+\sigma(2R)\big).
	\end{equation*}
	Up to relabeling, the last inequality can be written as
	\begin{align*}
		\phi(2^{-1}R)\le&\gamma\phi(R)+\sigma(R).
	\end{align*}
	By applying Lemma \ref{lem_vinogradoff}, we conclude that, for every $\mu\in(0,1)$, there exist $C=C(\gamma)>0$ and $\overline{\beta}=\overline{\beta}(\gamma,\mu)\in(0,1)$ such that
	\begin{equation}\label{4}
		\phi(r)\le C\left[\Big(\frac{r}{R}\Big)^{\overline{\beta}}\phi(R)+\sigma(r^{\mu}R^{1-\mu})\right].
	\end{equation}
	Let us consider $\sigma(r^{\mu}R^{1-\mu})$. By combining Proposition \ref{prop_Adams_Fournier} and the embedding $W^{2,d}(B_1)\hookrightarrow C(\overline{B_1})$, one gets
	\begin{align}\label{5}
		\int_{B_{r^{\mu}R^{1-\mu}}}|Du|^d\,{\rm d}x\le&\int_{B_{R}}|Du|^d\,{\rm d}x \notag\\
		\le&C\left(\int_{B_{R}}|D^2u|^d\,{\rm d}x+\int_{B_{R}}|u|^d\,{\rm d}x\right) \notag\\
		\le&C\left(\int_{B_{R}}|D^2u|^d\,{\rm d}x+R^d\|u\|_{L^\infty(B_1)}^d\right)\notag\\
        \leq& C \left( \phi(R)+ R^d\right).
	\end{align}
	Let $\beta\in(0,\overline{\beta})$; in particular, $\beta<d$. Combining \eqref{4} with \eqref{5}, up to relabeling the constants, we have
	\begin{equation*}
		\phi(r)\le C\left[\Big(\frac{r}{R}\Big)^{\overline{\beta}}+1\right]\phi(R)+C^\prime R^{\beta}.
	\end{equation*}
	From Lemma \ref{lem_stima119}, it follows that
	\begin{equation*}
		\phi(r)\le Cr^{\beta}.
	\end{equation*}
	
	\medskip
	
	\noindent{\bf Step 4 - }To complete the proof, we define $w_i:=u_{x_i}$, $i=1,\ldots,d$. Clearly,
	\[
		\int_{B_r}|Dw_i|^d{\rm d}x\leq \int_{B_r}|D^2u|^d{\rm d}x\leq Cr^\beta, \hspace{.3in}\forall r\in(0,R].
	\]
From Morrey's characterization of H\"older continuity (see Proposition \ref{morrey}), we conclude  
$$u_{x_i}\in C_{{\rm loc}}^{0,\alpha}(B_1), \qquad i=1,\ldots,d,$$
for $\alpha:=\beta/d \in(0,1)$, and the result follows.
 
\end{proof}

\begin{Remark} \label{tfa+yy}
We note that the proof of Theorem \ref{thm_ll} can be extended to minimizers of functionals $I:\mathcal{A}\to\mathbb{R}$ of the type
\begin{equation*}
         I(w)=\int_{B_1} \left[ F\left(D^2w\right) \right]^d\,{\rm d}x,
\end{equation*}
where $F:\mathbb{R}^{d^2}\to\mathbb{R}$ satisfies
\begin{equation*}\label{lip}
         \lambda|M|\le F(M)\le\Lambda|M|,
\end{equation*}
$$DF(M):M\ge C_1|M| \qquad \mathrm{and} \qquad |DF(M)|\le C_2,$$ 
for every $M\in\mathcal{S}(d)$, and some fixed constants $0<\lambda\le\Lambda$ and $C_1,C_2>0$. 
Under these assumptions, the proof of Theorem \ref{thm_ll} can be retraced in a completely analogous way. An example of such $F$ is given by
\begin{equation*}
    F(x,M)=a(x)|M|,
\end{equation*}
where $a\in C^\infty(B_1)\cap L^\infty(B_1)$ satisfies $a\ge\delta>0$ in $B_1$, for some fixed constant $\delta>0$.

 \end{Remark}

	\medskip
	
	{\small \noindent{\bf Acknowledgments.} All authors are partially supported by the Centre for Mathematics of the University of Coimbra (UIDB/00324/2020, funded by the Portuguese Government through FCT/MCTES). EP is partially supported by FAPERJ (grants E26/200.002/2018 and E26/201.390/2021). JMU is partially supported by the King Abdullah University of Science and Technology (KAUST).}


\medskip
	
	\begin{thebibliography}{99}
		
		\bibitem{Adams_Fournier}
		Robert Adams and John J. F. Fournier.
		\newblock {\em Sobolev spaces}, volume~140 of {\em Pure and Applied Mathematics}.
		\newblock Elsevier/Academic Press, Amsterdam. 2003. 
				
		\bibitem{Andrade_Pimentel2021}
		P\^edra Andrade and Edgard A. Pimentel.
		\newblock Stationary fully nonlinear mean-field games.
		\newblock {\em J. Anal. Math.}, 145(1):335--356, 2021.

		\bibitem{avilesgiga2}
		Patricio Aviles and Yoshikazu Giga.
		\newblock A mathematical problem related to the physical theory of liquid crystal configurations.
		\newblock In {\em Miniconference on geometry and partial differential equations, 2 ({C}anberra, 1986)}, volume~12 of {\em Proc. Centre Math. Anal. Austral. Nat. Univ.}, pages 1--16. Austral. Nat. Univ., Canberra, 1987.

		\bibitem{avilesgiga1}
		Patricio Aviles and Yoshikazu Giga.
		\newblock The distance function and defect energy.
		\newblock {\em Proc. Roy. Soc. Edinburgh Sect. A}, 126(5):923--938, 1996.

		\bibitem{kohn3}
		Jacob Bedrossian and Robert V. Kohn.
		\newblock Blister patterns and energy minimization in compressed thin films on
		compliant substrates.
		\newblock {\em Comm. Pure Appl. Math.}, 68(3):472--510, 2015.
		
		\bibitem{Bhattacharya_2022}
		Arunima Bhattacharya.
		\newblock Regularity for critical points of convex functionals on {H}essian spaces.
		\newblock {\em Proc. Amer. Math. Soc.}, 150(12):5217--5230, 2022.
		
		\bibitem{BiancaPimentelUrbano_2023}
		Vincenzo Bianca, Edgard A. Pimentel and Jos\'e Miguel Urbano.
		\newblock BMO-regularity for a degenerate transmission problem.
		\newblock {\em arXiv Preprint}, 2301.01171, p.26. 2023.
		
		\bibitem{ccbook}
		Luis A. Caffarelli and Xavier Cabr\'{e}.
		\newblock {\em Fully nonlinear elliptic equations}, volume~43 of {\em American
		Mathematical Society Colloquium Publications}.
		\newblock American Mathematical Society, Providence, RI, 1995.
		
		\bibitem{AliceChang1}
		Sun-Yung A.Chang, Lihe Wang, and Paul C. Yang.
		\newblock A regularity theory of biharmonic maps.
		\newblock {\em Comm. Pure Appl. Math.}, 52(9):1113--1137, 1999.

		\bibitem{AliceChang2}
		Sun-Yung A.Chang, Matthew J. Gursky, and Paul C. Yang.
		\newblock Regularity of a fourth-order nonlinear {PDE} with critical exponent.
		\newblock {\em Amer. J. Math.}, 121(2):215--257, 1999.
		
		\bibitem{Gilbarg-Trudinger}
		David Gilbarg and Neil S. Trudinger.
		\newblock {\em Elliptic partial differential equations of second order.}
		\newblock Classics in Mathematics, Reprint of the 1998 edition.
		\newblock Springer-Verlag, Berlin, 2001.
		
		\bibitem{GiOr97}
		Gustavo Gioia and Michael Ortiz.
		\newblock Delamination of compressed thin films.
		\newblock {\em Adv. Appl. Mech.}, 33:119-192, 1997.
		
		\bibitem{JaOtPe}
		Pierre-Emmanuel Jabin, Felix Otto, and Beno\^{i}t Perthame.
		\newblock Line-energy {G}inzburg-{L}andau models: zero-energy states.
		\newblock {\em Ann. Sc. Norm. Super. Pisa Cl. Sci. (5)}, 1(1):187--202, 2002.

		\bibitem{kohnicm}
		Robert V. Kohn.
		\newblock Energy-driven pattern formation.
		\newblock In {\em International {C}ongress of {M}athematicians. {V}ol. {I}}, pages 359--383. Eur. Math. Soc., Z\"{u}rich, 2007.

		\bibitem{kohn2}
		Robert V. Kohn and Ethan O'Brien.
		\newblock The wrinkling of a twisted ribbon.
		\newblock {\em J. Nonlinear Sci.}, 28(4):1221--1249, 2018.
		
		\bibitem{Oh93}
		Yong-Geun Oh.
		\newblock Volume minimization of {L}agrangian submanifolds under {H}amiltonian deformations.
		\newblock {\em Math. Z.}, 212(2):175--192, 1993.

        \bibitem{Modica-Mortola}
        Luciano Modica and Stefano Mortola.
        \newblock Un esempio di $\Gamma$-convergenza
        \newblock \emph{Boll. Un. Mat. Ital. B (5)}, 14(1):285--299, 1977.
		
		\bibitem{SW03}
		Richard Schoen and Jon Wolfson.
		\newblock The volume functional for {L}agrangian submanifolds.
		\newblock In {Lectures on partial differential equations, New Stud. Adv. Math.}, Vol. 2, 181--191, 2003. 
		\newblock Int. Press, Somerville, MA.

        \bibitem{venkataramani}
		Shankar C. Venkataramani.
		\newblock Lower bounds for the energy in a crumpled elastic sheet---a minimal ridge.
		\newblock {\em Nonlinearity}, 17(1):301--312, 2004. 
				
		\bibitem{Wang1}
		Changyou Wang.
		\newblock Biharmonic maps from {$\mathbb{R}^4$} into a {R}iemannian manifold.
		\newblock {\em Math. Z.}, 247(1):65--87, 2004.
		
		\bibitem{Wang2}
		Changyou Wang. 
		\newblock Stationary biharmonic maps from {$\mathbb{R}^m$} into a {R}iemannian manifold.
		\newblock {\em Comm. Pure Appl. Math.}, 57(4):419--444, 2004.

        \bibitem{Widman}
        Kjell-Ove Widman.
        \newblock H\"older continuity of solutions of elliptic equations.
        \newblock \emph{Manuscripta Mathematica}, 5:299--308, 1971.
        
    \end{thebibliography}
\end{document}